\documentclass[12pt,reqno]{article}
\usepackage{amsfonts,mathrsfs}
\usepackage{fancyhdr,amscd}
\usepackage{anysize}
\usepackage{setspace}
\usepackage{graphicx,epsfig}
\usepackage{amsthm,latexsym,amsmath,amssymb}
\usepackage[perpage,symbol*]{footmisc}
\usepackage[colorlinks=true]{hyperref}
\usepackage{xcolor}
\usepackage{amsfonts,amsmath,amssymb,latexsym}
\usepackage{mathrsfs}

\usepackage{epsfig,CJK,fancyhdr}
\usepackage{CJK,fancyhdr,amscd}
\newtheorem{thm}{Theorem}[section]
\newtheorem{defi}[thm]{Definition}
\newtheorem{prop}[thm]{Proposition}
\newtheorem{lem}[thm]{Lemma}

\newtheorem{eg}{Example}[section]
\newtheorem{rem}[thm]{Remark}

\def\pf{\noindent{\bf Proof.} }

\makeatletter \@addtoreset{equation}{section} \makeatother
\def\beq{\begin{equation}}
\def\eeq{\end{equation}}
\def\endproof{$\hfill\Box$}

\begin{document}
\baselineskip=20pt  \hoffset=-3cm \voffset=0cm \oddsidemargin=3.2cm
\evensidemargin=3.2cm \thispagestyle{empty}\vspace{10cm}

\hbadness=10000
\tolerance=10000
\hfuzz=150pt
\baselineskip=20pt  \hoffset=-3cm \voffset=0cm \oddsidemargin=3.2cm
\evensidemargin=3.2cm \thispagestyle{empty}\vspace{10cm}

\hbadness=10000
\tolerance=10000
\hfuzz=150pt
\title{\textbf{Relative Morse index theory and  applications in wave equations}}
\author{\Large Qi Wang $^{{\rm a,b}}$,$\quad$  Li Wu $^{{\rm b}}$}
\date{} \maketitle
\begin{center}
\it\scriptsize ${}^{\rm a}$ School of Mathematics and Statistics, Henan University, Kaifeng 475000, PR China\\
${}^{\rm b}$Department of Mathematics, Shandong University, Jinan, Shandong, 250100, PR China\\
\end{center}

\footnotetext[0]{$^a${\bf Corresponding author.} Supported by NNSF of China(11301148) and PSF of China(188576).}
\footnotetext[0]{\footnotesize\;{\it E-mail address}: Q.Wang@vip.henu.edu.cn. (Qi Wang), nankai.wuli@gmail.com (Li Wu).}


\noindent
{\bf Abstract:}  { We develop the relative Morse index theory for linear self-adjoint operator equation without compactness assumption and give the relationship between the index defined in \cite{Wang-Liu-2016} and \cite{Wang-Liu-2017}. Then we generalize the method of saddle point reduction and get some critical point theories by the index, topology degree and critical point theory.  As applications, we consider the existence and multiplicity of periodic solutions of wave equations.}



\section{Introduction}\label{section-introduction}
Many problems can be displayed as a self-adjoint operator equation
\[
 Au=F'(u),\;u\in D(A)\subset \mathbf H,\eqno{(O.E.)},
\]
where $\mathbf H$ is an infinite-dimensional separable Hilbert space, $A$ is a  self-adjoint operator on $\mathbf H$ with its domain $D(A)$, $F$ is a nonlinear functional on $\mathbf H$.
Such as boundary value problem for Laplace's equation on bounded domain, periodic solutions of Hamiltonian systems, Schr\"{o}dinger equation, periodic solutions of wave equation and so on.
By variational method, we know that the solutions of (O.E.) correspond to the critical points of a functional.
So we can transform the problem of finding the solutions of (O.E.) into the problem of finding the critical points of the functional.
 From 1980s, begin with Ambrosetti and Rabinowitz's famous work\cite{Ambrosetti-Rabinowitz-1973}(Mountain Pass Theorem), many crucial variational methods have been developed, such as Minimax-methods,  Lusternik-Schnirelman theory, Galerkin approximation methods, saddle point reduction methods, dual variational methods, convex analysis theory, Morse theory and so on (see\cite{Amann-1976},\cite{Amann-Zehnder-1980},\cite{Aubin-Ekeland-1984},\cite{Chang-1993},\cite{Ekeland-1990},\cite{Ekeland-Temam-1976} and the reference therein).

 We classified all of these variational problems into three kinds by the spectrum of $A$. For simplicity, denote by $\sigma(A)$, $\sigma_e(A)$ and $\sigma_d(A)$ the spectrum, the essential spectrum and the discrete finite dimensional point spectrum of $A$ respectively.

The first is $\sigma(A)=\sigma_d(A)$ and $\sigma(A)$ is bounded from below(or above), such as boundary value problem for Laplace's equation on bounded domain and periodic problem for second order Hamiltonian systems.
Morse theory can be used directly in this kind and  this is the simplest situation.

The second is $\sigma(A)=\sigma_d(A)$ and $\sigma(A)$ is unbounded from above and below, such as periodic problem for first order Hamiltonian systems.
In this kind, Morse theory cannot be used directly because in this situation the functionals are strongly indefinite and the Morse indices at the critical points of the functional are infinite. In order to overcome this difficulty, the index theory is worth to note here.  By the work  \cite{Ekeland-1984} of  Ekeland, an index theory for convex linear Hamiltonian systems was established.
By the works  \cite{Conley-Zehnder-1984,Long-1990,Long-1997,Long-Zehnder-1990} of Conley, Zehnder and Long, an index theory for symplectic paths was introduced.
These index theories have important and extensive applications, e.g \cite{Dong-Long-1997,Ekeland-Hofer-1985,Ekeland-Hofer-1987,Liu-Long-Zhu-2002,Long-Zhu-2000}.
In \cite{Zhu-Long-1999, Long-Zhu-2000-2} Long and Zhu defined spectral flows for paths of linear operators and redefined Maslov index for symplectic paths.
Additionally,  Abbondandolo defined the concept of relative Morse index theory for Fredholm operator with compact perturbation (see\cite{Abb-2001} and the references therein). In the study of the $L$-solutions (the solutions starting and ending at the same Lagrangian subspace $L$) of Hamiltonian systems, Liu in  \cite{Liu-2007} introduced an index theory for symplectic paths using the algebraic methods and gave some applications in
 \cite{ Liu-2007, Liu-2007-2}. This index had been generalized by Liu, Wang and Lin in  \cite{Liu-Wang-Lin-2011}.
In addition to the above index theories defined for specific forms, Dong in \cite{Dong-2010} developed an index theory for abstract operator equations (O.E.).

 The third  is $\sigma_e(A)\neq \emptyset$, the most complex situation. Since lack of compactness, many classical methods can not be used here. Specially, if $\sigma_e(A)\cap(-\infty,0)\neq \emptyset$ and $\sigma_e(A)\cap(0,\infty)\neq \emptyset$, Ding established a  series of critical points theories and applications in homoclinic orbits in Hamiltonian systems, Dirac equation,  Schr\"{o}dinger equation and so on, he named these problems???very strongly indefinite problems (see \cite{Ding-2007},\cite{Ding-2017}). Wang and Liu  defined the index theory ($i_A(B),\nu_A(B)$) for this kind and gave some applications in wave equation, homoclinic orbits in Hamiltonian systems and Dirac equation, the methods include dual variation and saddle point reduction(see \cite{Wang-Liu-2016} and\cite{Wang-Liu-2017}). Additionally, Chen and Hu in \cite{Chen-Hu-2007} defined the index for homoclinic orbits of Hamiltonian systems. Recently, Hu and Portaluri in \cite{Hu-Portaluri-2017} defined the index theory for heteroclinic orbits of Hamiltonian systems.\\

 In this paper, consider the kind of $\sigma_e(A)\neq \emptyset$. Firstly, we develop the relative Morse index theory.  Compared with Abbondandolo's work(\cite{Abb-2001}), we generalize the concept of relative Morse index $i^*_A(B)$ for Fredholm operator  without the compactness assumption on the  perturbation term(see Section \ref{section-relative Morse index}). And we gave the relationship between the relative Morse index $i^*_A(B)$ and the index  $i_A(B)$ defined in \cite{Wang-Liu-2016} and\cite{Wang-Liu-2017}. The bridge between them is the concept of spectral flow. As far as we know, the spectral flow is introduced by Atiyah-Patodi-Singer(see\cite{Atiyah-Patodi-Singer-1976}). Since then, many interesting properties and applications of spectral flow have been subsequently established(see\cite{Cappell-Lee-Miller-1994},\cite{Floer-1988},\cite{Robbin-Salamon-1993},\cite{Robbin-Salamon-1995} and \cite{Zhu-Long-1999}).

Secondly, we generalize the method of saddle point reduction and get some critical point theories. With the relative Morse index defined above, we will establish some new abstract critical point theorems by saddle point reduction, topology degree and Morse theory, where we do not need the nonlinear term to be $C^2$ continuous(see Section \ref{section-saddle point reduction}).

Lastly, as applications, we consider the existence and multiplicity of the periodic solutions for wave equation and give some new results(sec Section \ref{section-applications}). To the best of the authors' knowledge, the problem of finding periodic solutions of nonlinear wave equations has attracted much attention since 1960s. Recently, with critical point theory, there are many results on this problem.
For example, Kryszewski and Szulkin in \cite{Kryszewski-Szulkin-1997} developed an infinite dimensional cohomology theory and the corresponding Morse theory,
with these theories, they  obtained the existence of nontrivial periodic solutions of one dimensional wave equation.
Zeng, Liu and Guo in \cite{Zeng-Liu-Guo-2004}, Guo and Liu  in \cite{Gou-Liu-2007}  obtained the existence and multiplicity of nontrivial periodic solution of one dimensional wave equation
and beam equation by their Morse index theory developed in \cite{Guo-Liu-Zeng-2004}. Tanaka in \cite{Tanaka-2006}  obtained the existence of nontrivial periodic solution of
one dimensional wave equation by linking methods. Ji and Li in \cite{Ji-Li-2006} considered the periodic solution of one dimensional wave equation with $x$-dependent coefficients.
  By minimax principle, Chen and Zhang in \cite{Chen-Zhang-2014} and \cite{Chen-Zhang-2016}  obtained infinitely many symmetric periodic solutions
of $n$-dimensional wave equation. Ji in \cite{Ji-2018} considered the periodic solutions for one dimensional wave equation with bounded nonlinearity and $x$-dependent coefficients.

\section{Relative Morse Index $i^*_A(B)$ and the  relationship with $i_A(B)$}\label{section-relative Morse index}
Let $\mathbf H$ be an infinite dimensional separable Hilbert space with inner product $(\cdot,\cdot)_\mathbf H$ and norm $\|\cdot\|_\mathbf H$.
Denote by $\mathcal O(\mathbf H)$ the set of all linear self-adjoint operators on $\mathbf H$. For $A\in \mathcal O(\mathbf H)$, we denote by
$\sigma(A)$ the spectrum of $A$ and $\sigma_e(A)$ the essential spectrum of $A$. We define a subset of $\mathcal O( \mathbf H)$ as follows
\[
\mathcal O^0_e(a,b)=\{A\in \mathcal O(\mathbf H)|\;\sigma_e(A)\cap(a,b)=\emptyset \;{\rm and}\;\sigma(A)\cap (a, b)\ne \emptyset\}.
\]
Denote $\mathcal{L}_s(\mathbf H)$ the set of all  linear bounded self-adjoint operators on $\mathbf H$ and  a subset of $\mathcal{L}_s(\mathbf H)$ as follows
 \begin{equation}\label{eq-L0}
 \mathcal{L}_s(\mathbf H,a,b)=\{B\in \mathcal{L}_s(\mathbf H), \;a\cdot I<B< b\cdot I\},
\end{equation}
where $I$ is the identity map on $\mathbf H$, $B< b\cdot I$ means that there exists $\delta>0$ such that $(b-\delta)\cdot I-B$ is positive define,
 $B> a\cdot I$  has  the similar meaning. For any $B\in\mathcal{L}_s(\mathbf H,a,b)$, we have the index pair ($i_A(B),\nu_A(B)$)(see \cite{Wang-Liu-2016,Wang-Liu-2017} for details).
 In this section, we will define the relative Morse index  $i^*_A(B)$ and give the relationship with $i_A(B)$.

\subsection{Relative Morse Index $i^*_A(B)$}
As the beginning of this subsection, we will give a brief introduction of relative Morse index.  The relative Morse index can be derived in different ways (see\cite{Abb-2001,Chang-Liu-Liu-1997,Fei-1995,Zhu-Long-1999}). Such kinds of indices have been extensively studied in dealing with periodic orbits of first order Hamiltonian systems. As far as authors known, the existing relative Morse index theory can be regarded as compact perturbation for Fredholm operator. Assume $A$ is a self-adjoint Fredholm  operator on Hilbert space $\mathbf H$, with the  orthogonal splitting
\begin{equation}\label{eq-decomposition of space H}
\mathbf H=\mathbf H^-_A\oplus \mathbf H^0_A\oplus \mathbf H^+_A,
\end{equation}
where $A$ is negative, zero and positive definite on $\mathbf H^-_A,\;\mathbf H^0_A$ and $\mathbf H^+_A$ respectively. Let $P_A$ denote the orthogonal projection from $\mathbf H$ to $\mathbf H^-_A$. If the perturbation term $F$ is a compact self-adjoint operator on $\mathbf H$, then we have $P_A-P_{A-F}$ is compact and  $P_A:\mathbf H^-_{A-F}\to \mathbf H^-_A$ is a Fredholm operator and we can define the so called relative Morse index by the Fredholm index of $P_A:\mathbf H^-_{A-F}\to \mathbf H^-_A$.

Generally, if the operator $A$ is not Fredholm operator or the perturbation $F$ is not compact, $P_A:\mathbf H^-_{A-F}\to \mathbf H^-_A$ will not be Fredholm operator and the concept of relative Morse index will be meaningless,
but if the perturbation lies in the gap of $\sigma_{e}(A)$, that is to say $A\in \mathcal O^0_e(\lambda_a,\lambda_b)$ for some $\lambda_a,\lambda_b\in\mathbb{R}$ and the perturbation $B\in \mathcal{L}_s(\mathbf H,\lambda_a,\lambda_b)$, we can also defined the relative Morse index $i^*_A(B)$ and give the relationship with the index $i_A(B)$ defined in \cite{Wang-Liu-2017}. Firstly, we need two abstract lemmas.

\begin{lem} \label{Fredholm projection}
 Let $A:\mathbf H\rightarrow\mathbf H$ be a bounded self-adjoint operator.
 Let $W,V$ be closed spaces of $\mathbf H$.
 Denote the orthogonal projection $\mathbf H\rightarrow Y$ by $P_Y$ for any closed linear subspace $Y$ of $\mathbf H$. Assume that\\
\noindent (1).  $(Ax,x)_\mathbf H<-\epsilon_1 \|x\|^2_\mathbf H,\;\forall x\in W\backslash\{0\}$, with some constant $\epsilon_1>0$,\\
\noindent (2). $(Ax,x)_\mathbf H> 0,\; \forall x\in V^{\bot}\backslash\{0\} $,\\
\noindent (3). $ (Ax,y)_\mathbf H=0, \forall x\in V,y\in V^\bot$.\\
Then $P_V|_{W}$ is an injection and $P_V(W)$ is a closed  subspace of $\mathbf H$.
Furthermore, if we assume \\
\noindent (4). $(Ax,x)_\mathbf H\leq 0,\; \forall x\in V\backslash\{0\}$,\\
and there is a closed subspace $U$ of $W^\bot$ such that\\
\noindent (5).  $W^\bot/U$ is  finite dimensional,\\
\noindent (6).  $(Ax,x)_\mathbf H>0,\;\forall x\in U\setminus \{0\}$. \\
Then $P_V:W\rightarrow V$ and $P_W:V\rightarrow W$ are both  Fredholm operators and
\[
{\rm ind}(P_W:V\rightarrow W)=-{\rm ind}(P_V:W\rightarrow V).
\]
\end{lem}
\begin{proof}
Note that $\ker P_V|_{W}= \ker P_V\cap W =V^\bot \cap W$.
From condition (1) and (2), we have  $V^\bot \cap W=\{0\}$, so $P_V|_{W}$ is an injection.
For $x\in W$, from condition (2) and (3), we have
\begin{align*}
-\|A\|\|P_Vx\|^2_\mathbf H&\leq(AP_V x,P_V x)_\mathbf H\\
                          &=(Ax,x)_\mathbf H-(A(I-P_V)x,(I-P_V)x)_\mathbf H\\
                          &\leq (Ax,x)_\mathbf H\\
                          &< -\epsilon_1 \|x\|^2_\mathbf H
\end{align*}
It follows that
\begin{equation}\label{eq-relative Morse index-3}
\|P_V x\|_\mathbf H\ge \sqrt{\frac{\epsilon_1}{\|A\|}}\|x\|_\mathbf H,\;\forall x\in W,
\end{equation}
so  $P_V(W)$ is a closed subspace of $\mathbf H$.

For any $x\in (P_V(W))^\bot \cap V $, that is to say $x\bot P_V(W)$ and $x\bot (I-P_V)(W)$, so we have $x\bot W$ and
\begin{equation}\label{eq-relative Morse index-1}
P_V(W)^\bot \cap V \subset W^\bot.
\end{equation}
From condition (4) and (6),
\begin{align}\label{eq-relative Morse index-2}
((P_V(W))^\bot \cap V) \cap U&\subset V\cap U\nonumber\\
                             & =\{0\}.
\end{align}
From \eqref{eq-relative Morse index-1}, \eqref{eq-relative Morse index-2} and condition (5),  $(P_V(W))^\bot \cap V$ is finite dimensional.
It follows that $P_V:W\rightarrow V$ is a Fredholm operator.
From \eqref{eq-relative Morse index-3}, we have
\begin{align*}
\|(I-P_V)x\|^2&=\|x\|^2-\|P_V x\|^2\\
              &\leq (1-\epsilon_1/\|A\|)\|x\|^2,\;\forall x\in W.
\end{align*}
It follows that $\|I-P_V|_W\|<1$.
So the operator $P_WP_V=P_W-P_W(I-P_V):W\rightarrow W$ is invertible.
It follows that $P_W:V\rightarrow W$ is surjective, and
\begin{equation}\label{eq-relative Morse index-4}
\ker P_W\cap P_V(W)={0}.
\end{equation}
Note that $V$ has the following decomposition
\[
V=P_V(W)\bigoplus ((P_V(W))^\bot \cap V),
\]
from \eqref{eq-relative Morse index-4} and  $\dim((P_V(W))^\bot \cap V)<\infty$, we have $\ker P_W \cap V$ is finite dimensional.
So the operator $P_W:V\rightarrow W$ is a Fredholm operator.
Since $P_WP_V:W\rightarrow W$ is invertible,  we have
\begin{align*}
0&={\rm ind}(P_WP_V:W\rightarrow W)\\
 &={\rm ind}(P_W:V\rightarrow W)+{\rm ind}(P_V:W\rightarrow V).
\end{align*}
Thus we have proved the lemma.
\end{proof}
\begin{lem}\label{finite_pertubation}
Let $V_1\subset V_2$,$ W_1\subset W_2$ be linear closed subspaces of $\mathbf H$ such that $V_2/V_1$ and $W_2/W_1$ are finite dimensional linear spaces.
Let $P_{V_i}$, $P_{W_j}$ be the orthogonal projections onto $V_i$ and $W_j$ and respectively, $i,j=1,2$.
Assume that $P_{W_{j^*}}:V_{i^*}\rightarrow W_{j^*}$ is a Fredholm operator for some fixed $i^*,j^*\in \{1,2\}$.
Then   $P_{W_j}:V_i \rightarrow W_j$, $i,j=1,2$ are all Fredholm operators.
Furthermore, we have
\begin{align*}
{\rm ind}(P_{W_j}:V_i\rightarrow W_j)=&{\rm ind}(P_{V_{i^*}}:V_i\rightarrow V_{i^*})+{\rm ind}(P_{W_{j^*}}:V_{i^*}\rightarrow W_{j^*})\\
&+{{\rm ind}}(P_{W_j}:W_{j^*}\rightarrow W_j).
\end{align*}
\end{lem}
\begin{proof}
Since  $V_2/V_1$ and $W_2/W_1$ are finite dimensional linear spaces, $P_{W_j}-P_{W_{j^*}}$ and $P_{V_i}-P_{V_{i^*}}$ are both compact operator. So $P_{W_j}P_{V_i}-P_{W_{j^*}}P_{V_{i^*}}$ is also compact operator.
Note that on $V_i$,
\[
(P_{W_j}-P_{W_j}P_{W_{j^*}}P_{V_{i^*}})|_{V_i}=P_{W_j}(P_{W_j}P_{V_i}-P_{W_{j^*}}P_{V_{i^*}})|_{V_i}.
\]
It follows that $P_{W_j}-P_{W_j}P_{W_{j^*}}P_{V_{i^*}}:V_i\rightarrow W_j$ is compact.
Then we can conclude that
\begin{align*}
{{\rm ind}}(P_{W_j}:V_i\rightarrow W_j)=&{\rm ind}(P_{W_j}P_{W_{j^*}}P_{V_{i^*}}:V_i\rightarrow W_j)\\
                                       =&{\rm ind}(P_{V_{i^*}}:V_i\rightarrow V_{i^*})+{\rm ind}(P_{W_{j^*}}:V_{i^*}\rightarrow W_{j^*})\\
                                        &+{{\rm ind}}(P_{W_j}:W_{j^*}\rightarrow W_j).
\end{align*}
We have proved the lemma.
\end{proof}

With these two lemmas, we can define the relative Morse index.  We consider a normal type that is  $A\in \mathcal O^0_e(-1,1)$ for simplicity.
Let $B\in \mathcal{L}_s(\mathbf H,-1,1)$ with its norm $\|B\|=c_B$, so we have $0\leq c_B<1$. Then $A-tB$ is a self-adjoint Fredholm operator for $t\in [0,1]$. We  have  $\sigma_{ess}(A-tB)\cap (-1+tc_B,1-tc_B)=0$. Let $E_{A-tB}(z)$ be the spectral measure of $A-tB$. Denote
\begin{equation}\label{eq-projection splitting}
P(A-tB,U)= \int_{U} dE_{A-tB}(z),
\end{equation}
 with $U\subset \mathbb{R}$, and rewrite it as $P(t,U) $ for simplicity. Let
 \[
 V(A-tB,U)={\rm im} P(t,U)
 \]
  and rewrite it as $V(t,U)$ for simplicity. For any $c_0\in\mathbb{R}$ satisfying $c_B<c_0<1$,  we have
 \[
 ((A-B)x,x)_\mathbf H>(c_0-c_B)\|x\|^2_\mathbf H,\;x\in V(0,(c_0,+\infty))\cap D(A).
 \]
So there is $\epsilon >0$, such that
\begin{equation}\label{eq-relative Morse index-5}
((A-B)x,x)_\mathbf H>\epsilon ((|A-B|+I)x,x)_\mathbf H ,\forall x\in V(0,(c_0,+\infty))\cap D(A)
\end{equation}
Similarly, we have
\begin{equation}\label{eq-relative Morse index-6}
((A-B)x,x)_\mathbf H<-\epsilon ((|A-B|+I)x,x)_\mathbf H ,\forall x\in V(0,(-\infty,-c_0))\cap D(A)
\end{equation}
Denote
\[
P_{s,a}^{t,b}:=P(t,(-\infty,b))|_{V(s,(-\infty,a))},\forall t,s\in [0,1]\;{\rm and}\;a,b\in \mathbb{R}.
\]
Clearly, we have $P(t,(-\infty,b))=P_{s,+\infty}^{t,b}$ , $\forall s\in [0,1] $.
\begin{lem}\label{inverse_relation}
For any $a\in[-c_0,c_0]$, the map $P_{0,a}^{1,0}$ $P_{1,0}^{0,a}$ are both Fredholm operators.
 Furthermore, we have ${\rm ind}(P_{0,a}^{1,0})=-{\rm ind}(P_{1,0}^{0,a})$.
\end{lem}
\begin{proof}
 From \eqref{eq-relative Morse index-5} and \eqref{eq-relative Morse index-6}, there is $\epsilon>0$ such that
  \[
  ((A-B)(|A-B|+I)^{-1}x,x)>\epsilon \|x\|^2,\;\forall x\in V(0,(c_0,+\infty)),
 \]
 and
 \[
  ((A-B)(|A-B|+I)^{-1}x,x)<-\epsilon \|x\|^2,\;\forall x\in V(0,(-\infty,-c_0)).
 \]
 Now, let the operator $(A-B)(|A-B|+I)^{-1}$, the spaces $V(0,(-\infty,-c_0))$, $V(1,(-\infty,0])$ and $V(0,(c_0,+\infty))$ be the operator $A$ and the spaces $W,V$ and $U$ in Lemma \ref{Fredholm projection} correspondingly.
It's easy to verify that condition (1), (2), (3), (4) and (6) are satisfied, and since $A\in \mathcal O^0_e(-1,1)$, $V(0,[-c_0,c_0])$ is finite dimensional, so condition (5) is  satisfied. Then $P_{0,-c_0}^{1,0} $ and $P_{1,0}^{0,-c_0} $ are both Fredholm operators.
We also have
\[
{\rm ind}(P_{0,-c_0}^{1,0})=-{\rm ind}(P_{1,0}^{0,-c_0}).
 \]
By Lemma \ref{finite_pertubation}, $ P_{0,a}^{1,0}$ and $P_{1,0}^{0,a}$ are both Fredholm operators with $a\in [-c_0,c_0]$, and we have
\[
 {\rm ind}(P_{0,a}^{1,0})=-{\rm ind}(P_{1,0}^{0,a}), a\in [-c_0,c_0].
  \]
\end{proof}

\begin{rem}\label{inverse_relation(general result)}
 Generally, we have $P_{s,a}^{t,b}\; and\; P_{t,b}^{s,a}$ are both  Fredholm operators
 with $a\in (-1+sc_B,1-sc_B)$, $b\in (-1+tc_B,1-tc_B)$ and we have
 \[
 {{\rm ind}}(P_{s,a}^{t,b})=-{{\rm ind}}(P_{t,b}^{s,a}).
 \]
 \end{rem}
Here we replace $A,B$ by $A'=A-sB$ and $B'=(t-s)B$ respectively in Lemma \ref{inverse_relation}, then all the proof will be same, so we omit the proof here.

\begin{defi}\label{defi-index defined by relative Morse index}
Define the relative Morse index by
\[
i^*_A(B):={\rm ind}(P^{0,0}_{1,0}),\;\forall B\in\mathcal{L}_s(\mathbf H,-1,1).
\]
\end{defi}

\subsection{The relationship between $i^*_A(B)$ and $i_A(B)$}
Now, we will prove that $i^*_A(B)=i_A(B)$ by the concept of spectral flow. We need some preparations. There are some equivalent definitions of spectral flow. We use the Definition 2.1, 2.2 and 2.6 in \cite{Zhu-Long-1999}.
Let $A_s$ be a path of self-adjoint Fredholm operators.
The APS projection  of $A_s$ is defined by $Q_{A_s}=P(A_s,[0,+\infty))$ .
Recall that locally, the spectral flow of $A_s$  is  the s-flow of $Q_{A_s}$.
Choose $\epsilon>0$ such that $V(A_{s_0},[0,+\infty))=V(A_{s_0},[-\epsilon,+\infty))$.
Then $\epsilon \notin \sigma (A_{s_0})$.
Let $P_{A_s}=P(A_s,[-\epsilon,+\infty))$.
Then there is $\delta>0$ such that $P_{A_s}$ is continuous on $(s_0-\delta,s_0+\delta)$
and $P_{A_s}-Q_{A_s}$ is compact for $s\in(s_0-\delta,s_0+\delta)$ .
The the s-flow of $Q_{A_s}$ on $[s_0,b]\subset (s_0-\delta,s_0+\delta)$  can be calculated as
\begin{align*}
sfl(Q_{A_s},[s_0,b])=& -{\rm ind}(P_{A_{s_0}}:V(A_{s_0},[0,\infty)\to V(A_{s_0},[0,\infty))\\
&+{\rm ind}(P_{A_{s_b}}:V(A_{s_b},[0,\infty)\to V(A_{s_b},[-\epsilon,\infty))\\
=&-{\rm dim}(V(A_{s_b},[-\epsilon,0)))\\
=&{\rm ind}({\rm Id}-P_{A_{s_b}}:V(A_{s_b},(-\infty,-\epsilon)\to V(A_{s_b},(-\infty,0)).
\end{align*}
If  $A_s=A-sB$, with $\epsilon$ and $\delta$ chosen like above,
we have $sf\{A-sB,[s_0,s_1]\}={\rm ind}P_{s_1,-\epsilon}^{s_1,0}$ for $[s_1,s_2]\subset [s_1,s_1+\delta]$.

\begin{lem}\label{continuity}
Let $t_0\in [0,1]$.
Let $a\in (-1+t_0c_B,1-t_0c_B)\backslash\sigma(A-t_0B) $.
Then we have
\[
 \lim_{s\to t_0}\|P_{t_0,a}^{t,0}-P_{s,a}^{t,0}P_{t_0,a}^{s,a}\|=0,\forall t\in[0,1].
\]
and
\begin{align*}
\lim_{s\to t_0}{{\rm ind}}(P_{s,a}^{t,0}))={{\rm ind}}(P_{t_0,a}^{t,0})\\
\lim_{s\to t_0}{{\rm ind}}(P_{t,0}^{s,a}))={{\rm ind}}(P_{t,0}^{t_0,a}).
\end{align*}
\end{lem}
\begin{proof}
Since $a\notin\sigma(A-t_0B) $,
there is $\delta_1>0$ such that $P(\cdot,(-\infty,a))$ is a continuous path of operators on $(t_0-\delta_1,t_0+\delta_1)$,
and
\[\|(P(s,(-\infty,a))-P(t_0,(-\infty,a)))\|<1\]
with $s\in (t_0-\delta_1,t_0+\delta_1)$.
Then $P_{t_0,a}^{s,a}$ and $ P_{s,a}^{t_0,a}$ are both homeomorphisms.
Note that on $V(t_0,(-\infty,a))$, we have
 \[
 P_{t_0,a}^{t,0}-P_{s,a}^{t,0}P_{t_0,a}^{s,a}
 =P(t,(-\infty,0))(P(t_0,(-\infty,a))-P(s,(-\infty,a)))|_{V(t_0,(-\infty,a))}.
 \]
 By the continuity of $P(s,(-\infty,a))$, it follows that
 \[
 \lim_{s\to t_0}\|P_{t_0,a}^{t,0}-P_{s,a}^{t,0}P_{t_0,a}^{s,a}\|=0.
\]
Then we have
\begin{align*}
{\rm ind}(P_{t_0,a}^{t,0})&=\lim_{s\to t_0}{\rm ind}(P_{s,a}^{t,0}P_{t_0,a}^{s,a})\\
                          &=\lim_{s\to t_0}{\rm ind}(P_{s,a}^{t,0})+{\rm ind}(P_{t_0,a}^{s,a}))\\
                          &=\lim_{s\to t_0}{{\rm ind}}(P_{s,a}^{t,0}).
\end{align*}
By remark \ref{inverse_relation(general result)}, we get
\begin{align*}
{\rm ind}(P_{t,0}^{t_0,a})&=-{\rm ind} (P_{t_0,a}^{t,0})\\
                          &=-\lim_{s\to t_0}{\rm ind}(P_{s,a}^{t,0})\\
                          &=\lim_{s\to t_0}{\rm ind}(P_{t,0}^{s,a}).
\end{align*}
\end{proof}

\begin{lem}\label{local_flow}
For each $t_1\in [0,1]$, there is $\delta >0$ such that
\[{\rm ind} (P_{t_1,0}^{t_2,0})=sf\{A-t_1 B-s(t_2-t_1)B,[0,1]\}\] with $|t_2-t_1|<\delta$.
\end{lem}
\begin{proof}
Since $A-t_1B$ is a Fredholm operator, there is $\epsilon>0$ such that $P(t_1,(-\infty,0))=P(t_1,(-\infty,-\epsilon))$.
It follows that $\epsilon \notin \sigma (A-t_1B)$, and we have
\[
P_{t_1,-\epsilon}^{t_2,0}=P_{t_1,0}^{t_2,0}.
\]
By lemma \ref{continuity}  we have
\[
\lim_{t_2\to t_1}{\rm ind}(P_{t_1,-\epsilon}^{t_2,-\epsilon})={\rm ind}(P_{t_1,-\epsilon}^{t_1,-\epsilon})=0.
\]
It follows that
\begin{align*}
\lim_{t_2\to t_1}{\rm ind}(P_{t_1,-\epsilon}^{t_2,0})&=\lim_{t_2\to t_1}{\rm ind}(P_{t_2,-\epsilon}^{t_2,0}P_{t_1,-\epsilon}^{t_2,-\epsilon})\\
&=\lim_{t_2\to t_1}{\rm ind}(P_{t_2,-\epsilon}^{t_2,0})+\lim_{t_2\to t_1}{\rm ind}(P_{t_1,-\epsilon}^{t_2,-\epsilon})\\
&=\lim_{t_2\to t_1}{\rm ind}(P_{t_2,-\epsilon}^{t_2,0}).
\end{align*}
So there is $\delta>0$ such that ${{\rm ind}}(P_{t_1,-\epsilon}^{t_2,0})={\rm ind}(P_{t_2,-\epsilon}^{t_2,0}) $ with $|t_2-t_1|<\delta$,
and  $P(t,(-\infty,-\epsilon))$ is continuous on $(t_1-\delta,t_1+\delta)$.
Note that $sf\{A-t_1 B-s(t_2-t_1)B,[0,1]\}={{\rm ind}} P_{t_2,-\epsilon}^{t_2,0}$ by continuation of $P(t,(-\infty,-\epsilon))$.
Then the lemma follows.
\end{proof}
\begin{lem}\label{additional}
${{\rm ind}}(P_{0,0}^{1,0})={{\rm ind}}(P_{t,0}^{1,0})+{{\rm ind}}( P_{0,0}^{t,0})$ with $\forall t\in [0,1]$.
\end{lem}
\begin{proof}
By Lemma \ref{finite_pertubation} and Lemma \ref{inverse_relation(general result)}, for any $t_0\in [0,1]$
\begin{align*}
{\rm ind}(P_{t_0,0}^{1,0})+{\rm ind}( P_{0,0}^{t_0,0})&={\rm ind}(P_{t_0,a}^{1,0})+{\rm ind}( P_{t_0,0}^{t_0,a})+{\rm ind}( P_{t_0,a}^{t_0,0})+{\rm ind}( P_{0,0}^{t_0,a})\\
                                                  &={\rm ind}(P_{t_0,a}^{1,0})+{{\rm ind}}( P_{0,0}^{t_0,a}),\; \forall a\in (-1+sc_B,1-sc_B).
\end{align*}

Choose $a_{t_0}\in (-1+t_0c_B,1-t_0c_B)$ and  $a_{t_0}\notin \sigma(A-t_0B)$.
By lemma \ref{continuity},
\[
f:t\to {\rm ind}(P_{t,a_{t_0}}^{1,0})+{\rm ind}( P_{0,0}^{t,a_{t_0}})
\]
 is continuous at $t_0$.
So the function $f:t\to {{\rm ind}}(P_{t,0}^{1,0})+{{\rm ind}}( P_{0,0}^{t,0}) $ is continuous on $[0,1]$.
So it must be a constant function.
It follows that
\[
{{\rm ind}}(P_{0,0}^{1,0})=f(1)=f(t)={{\rm ind}}(P_{t,a}^{1,0})+{{\rm ind}}( P_{0,0}^{t,a}).
\]
\end{proof}
\begin{rem}
In fact, we have
\[{{\rm ind}}(P_{a,0}^{b,0})={{\rm ind}}(P_{s,0}^{b,0})+{{\rm ind}}( P_{a,0}^{s,0})\] with $s\in [0,1]$.
\end{rem}
\begin{thm}\label{thm-relative morse index and spectral flow}
We have
\[
sf\{A-tB,[a,b]\}={{\rm ind}}(P_{a,0}^{b,0})=-{{\rm ind}}(P_{b,0}^{a,0})
\]
with $[a,b]\subset [0,1]$.
\end{thm}
\begin{proof}
It is a direct consequence of Lemma \ref{local_flow} and Lemma \ref{additional}.
\end{proof}

Now by the property of $i_A(B)$(see \cite[Lemma 2.9]{Wang-Liu-2016},\cite[Lemma 2.3]{Wang-Liu-2017}) and Theorem \ref{thm-relative morse index and spectral flow}, we have the following result.
\begin{prop}\label{prop-relations between indexes}
$
i^*_A(B)=i_A(B), A\in\mathcal{O}^0_e(-1,1),B\in\mathcal{L}_s(\mathbf H,-1,1).
$
\end{prop}
Generally, with the same method we can define the relative Morse index $i^*_A(B)$ for $A\in\mathcal{O}^0_e(\lambda_a,\lambda_b)$, $B\in\mathcal{L}_s(\mathbf H,\lambda_a,\lambda_b)$ and we can prove the index $i^*_A(B)$ coincide with $i_A(B)$ by the concept of spectral flow, we omit them here.


\section{Saddle point reduction of (O.E.) and some abstract critical points Theorems}\label{section-saddle point reduction}


Now for simplicity, let $b>0$ and $a=-b$, for $A\in\mathcal O^0_e(-b,b)$, we consider the following operator equation
\[
Az=F'(z),\;z\in D(A)\subset \mathbf H, \eqno(O.E.)
\]
where  $F\in C^1(\mathbf H,\mathbb{R})$. Assume\\
\noindent($F_1$) $F\in C^1(\mathbf H,\mathbb{R})$, $F':\mathbf H\to\mathbf H$ is Lipschitz continuous
 \begin{equation}\label{eq-the  Lipschitz continuity of F'}
\|F'(z+h)-F'(z)\|_{\mathbf H}\leq l_F\|h\|_{\mathbf H},\;\forall z,h\in\mathbf H,
\end{equation}
with its Lipschitz constant $l_F<b$.
\subsection{Saddle point reduction of (O.E.)}
In this part, assume $A\in\mathcal{O}^0_e(-b,b)$  and $F$ satisfies condition ($F_1$), we will consider the method of saddle point reduction without assuming the nonlinear term $F\in C^2(D(|A|^{1/2}))$, then we will give some abstract critical point theorems. Let $E_A(z)$ the spectrum measure of $A$, since $\sigma_e(A)\cap(-b,b)=\emptyset$, we can choose $l\in (l_F,b)$, such that
 \[
 -l, l\notin\sigma(A).
 \]
Different from the above section,  in this section, consider projection map $P(A,U)$ defined in \eqref{eq-projection splitting}  on $H$, for simplicity, we rewrite them  as
\begin{equation}\label{eq-projections}
P^-_A:=P(A,(-\infty,-l)),\;P^+_A:=P(A,(l,\infty)),\;P^0_A:=P(A,(-l,l)),
\end{equation}
in this section.
Then we have the following decomposition which is different from \eqref{eq-decomposition of space H},
\[
\mathbf H=\widehat{\mathbf  H}^-_A\oplus\widehat{\mathbf  H}^+_A \oplus\widehat{\mathbf  H}^0_A,
\]
where $\widehat{\mathbf  H}^*_A:=P^*_A\mathbf  H$($*=\pm, 0$) and $\widehat{\mathbf  H}^0_A$ is finite dimensional subspace of $\mathbf  H$,
for simplicity we rewrite $\mathbf H^*:=\widehat{\mathbf  H}^*_A$. Denote $A^*$ the restriction of $A$ on $\mathbf H^*$($*=\pm, 0$), thus we have $(A^\pm)^{-1}$ are bounded self-adjoint linear operators on $\mathbf H^\pm$ respectively and satisfying
\begin{equation}\label{eq-the norm of the inverse of Apm}
\|(A^\pm)^{-1}\|\leq \frac{1}{l}.
\end{equation}
Then (OE) can be rewritten as
\begin{equation}\label{eq-decomposition 1 of OE}
z^\pm =(A^\pm)^{-1}P^\pm_A F'(z^++z^-+z^0),
\end{equation}
and
\begin{equation}\label{eq-decomposition 2 of OE}
A^0z^0=P^0_AF'(z^++z^-+z^0),
\end{equation}
where $z^*=P^*_Az$($*=\pm, 0$), for simplicity, we rewrite $x:=z^0$. From \eqref{eq-the  Lipschitz continuity of F'} and \eqref{eq-the norm of the inverse of Apm}, we have $(A^\pm)^{-1}P^\pm_A F'$ is  contraction map on $\mathbf H^+\oplus \mathbf H^-$ for any $x\in\mathbf H^0$. So there is a map $z^\pm(x):\mathbf H^0\to\mathbf H^\pm$ satisfying
\begin{equation}\label{eq-zpm(x)}
z^\pm(x)=(A^\pm)^{-1}P^\pm F'(z^\pm(x)+x),\;\forall x\in\mathbf H,
\end{equation}
and  the following properties.
\begin{prop}\label{prop-continuous and property of saddle point reduction}
(1) The map $z^\pm(x):\mathbf H^0\to \mathbf H^\pm$ is continuous, in fact we have
\[
\|(z^++z^-)(x+h)-(z^++z^-)(x)\|_\mathbf H\leq\frac{l_F}{l-l_F}\|h\|_{\mathbf H},\;\;\forall x,h\in\mathbf H^0.
\]
(2) $\|(z^++z^-)(x)\|_\mathbf H\displaystyle\leq \frac{l_F}{l-l_F}\|x\|_{\mathbf H}+\frac{1}{l-l_F}\|F'(0)\|_{\mathbf H}$.
\end{prop}
\noindent\textbf{Proof.}(1) For any $x,\;h\in \mathbf H^0$, here we write $z^\pm(x):=z^+(x)+z^-(x)$ and $(A^\pm)^{-1}P^\pm_A:=(A^+)^{-1}P^+_A+(A^-)^{-1}P^-_A$ for simplicity, we have
\begin{align*}
\|z^\pm(x+h)-z^\pm(x)\|_{\mathbf H} &=\|(A^\pm)^{-1}P^\pm_A F'(z^\pm(x+h)+x+h)-(A^\pm)^{-1}P^\pm_A F'(z^\pm(x)+x)\|_{\mathbf H}\\
                             &\leq\frac{1}{l}\|F'(z^\pm(x+h)+x+h)-F'(z^\pm(x)+x)\|_{\mathbf H}\\
                             &\leq\frac{l_F}{l}\|z^\pm(x+h)-z^\pm(x)+h\|_{\mathbf H}\\
                             &\leq\frac{l_F}{l}\|z^\pm(x+h)-z^\pm(x)\|_{\mathbf H}+\frac{l_F}{l}\|h\|_{\mathbf H}.
\end{align*}
So we have $\|z^\pm(x+h)-z^\pm(x)\|_\mathbf H\leq\frac{l_F}{l-l_F}\|h\|_{\mathbf H}$ and the map $z^\pm(x):{\mathbf H}^0\to {\mathbf H}^\pm$ is continuous.\\
(2)Similarly,
\begin{align*}
\|z^\pm(x)\|_{\mathbf H}&=\|(A^\pm)^{-1}P^\pm_A F'(z^\pm(x)+x)\|_{\mathbf H}\\
                &\leq\frac{1}{l}\|F'(z^\pm(x)+x)\|_{\mathbf H}\\
                &\leq\frac{1}{l}\|F'(z^\pm(x)+x)-F'(0)\|_{\mathbf H}+\frac{1}{l}\|F'(0)\|_{\mathbf H}\\
                &\leq\frac{l_F}{l}(\|z^\pm(x)\|_{\mathbf H}+\|x\|_H)+\frac{1}{l}\|F'(0)\|_{\mathbf H}.
\end{align*}
So we have $\|z^\pm(x)\|_\mathbf H\leq \frac{l_F}{l-l_F}\|x\|_{\mathbf H}+\frac{1}{l-l_F}\|F'(0)\|_{\mathbf H}$.\endproof

\begin{rem}Denote $\mathbf E=D(|A|^{\frac{1}{2}})$, with its norm
\[
\|z\|^2_\mathbf E:=\||A|^{\frac{1}{2}}(z^++z^-)\|^2_\mathbf H+\|x\|^2_\mathbf H,\;u\in \mathbf E.
\]
From \eqref{eq-zpm(x)}, we have $z^\pm(x)\in D(A)\subset\mathbf E$,
and we have\\
 (1) The map $z^\pm(x):\mathbf H^0\to \mathbf E$ is continuous, and
 \begin{equation}\label{eq-uniform continuous of z in E}
 \|(z^++z^-)(x+h)-(z^++z^-)(x)\|_\mathbf E\leq\frac{l_F \cdot l^\frac{1}{2}}{l-l_F}\|h\|_{\mathbf H},\;\;\forall x,h\in\mathbf H^0.
 \end{equation}
(2) $\|(z^++z^-)(x)\|_\mathbf E\displaystyle\leq \frac{l^\frac{1}{2}}{l-l_F}(l_F\cdot\|x\|_{\mathbf H}+\|F'(0)\|_{\mathbf H})$.
\end{rem}
\noindent{\bf Proof.} The proof is similar to Proposition \ref{prop-continuous and property of saddle point reduction}, we only prove (1).
\begin{align*}
\|z^\pm(x+h)-z^\pm(x)\|_{\mathbf E} &=\|(|A|^{\frac{1}{2}})[z^\pm(x+h)-z^\pm(x)]\|_{\mathbf H}\\
                                    &=\|(A^\pm)^{-\frac{1}{2}}[P^\pm_A F'(z^\pm(x+h)+x+h)-P^\pm_A F'(z^\pm(x)+x)]\|_{\mathbf H}\\
                                    &\leq\frac{1}{l^\frac{1}{2}}\|F'(z^\pm(x+h)+x+h)-F'(z^\pm(x)+x)\|_{\mathbf H}\\
                                    &\leq\frac{l_F}{l^\frac{1}{2}}\|z^\pm(x+h)-z^\pm(x)+h\|_{\mathbf H}\\
                                    &\leq\frac{l_F}{l^\frac{1}{2}}\|z^\pm(x+h)-z^\pm(x)\|_{\mathbf H}+\frac{l_F}{l^\frac{1}{2}}\|h\|_{\mathbf H}\\
                                    &\leq\frac{l_F}{l}\|z^\pm(x+h)-z^\pm(x)\|_{\mathbf E}+\frac{l_F}{l^\frac{1}{2}}\|h\|_{\mathbf H},
\end{align*}
where the last inequality depends on the fact that $\|z^\pm\|_\mathbf E\geq l^\frac{1}{2}\|z^\pm\|_\mathbf H$, so we have \eqref{eq-uniform continuous of z in E}.

Now, define the map $z:\mathbf H^0\to\mathbf H$ by
\[
z(x)=x+z^+(x)+z^-(x).
\]
 Define the functional $a:\mathbf H^0\to \mathbb{R}$ by
 \begin{equation}\label{eq-saddle point reduction}
 a(x)=\frac{1}{2}(Az(x),z(x))_\mathbf H-F(z(x)),\;x\in \mathbf H^0.
 \end{equation}
 With standard discussion, the critical points of $a$ correspond to the solutions of (O.E.), and we have
 \begin{lem}\label{lem-the smoothness of a}
 Assume $F$ satisfies ($F_1$), then we have $a\in C^1(\mathbf H^0,\mathbb{R})$ and
 \begin{equation}\label{eq-saddle point reduction-the derivative of a}
 a'(x)=Az(x)-F'(z(x)),\;\;\forall x\in \mathbf H^0.
 \end{equation}
 Further more, if $F\in C^2(\mathbf H,\mathbb R)$, we have $a\in C^2(\mathbf H^0,\mathbb{R})$, for any critical point $x$ of $a$, $F''(z(x))\in \mathcal{L}_s(\mathbf H,-b,b)$  and the morse index $m^-_a(x)$  satisfies the following equality
 \begin{equation}\label{eq-relation between morse index and our index}
 m^-_a(x_2)-m^-_a(x_1)=i^*_A(F''(z(x_2)))-i^*_A(F''(z(x_1))),\;\;\forall x_1,x_2\in \mathbf H^0.
 \end{equation}
  \end{lem}
  \pf For any $x,h\in\mathbf H^0$, write
  \[
  \eta(x,h):=z^+(x+h)+z^-(x+h)-z^+(x)-z^-(x)+h
  \]
   for simplicity, that is to say
  \[
  z(x+h)=z(x)+\eta(x,h),\;\;\forall x,h\in\mathbf H^0,
  \]
  and from \eqref{eq-uniform continuous of z in E}, we have
  \begin{equation}\label{eq-uniform astimate of eta}
  \|\eta(x,h)\|_\mathbf H\leq C\|h\|_{\mathbf H},\;\;\forall x,h\in\mathbf H^0,
  \end{equation}
  where $C=\displaystyle \frac{l+l_F}{l-l_F}$. Let $h\to 0$ in $\mathbf H^0$, and for any $x\in\mathbf H^0$, we have
  \begin{align*}
  a(x+h)-a(x)=&\frac{1}{2}[(Az(x+h),z(x+h))_\mathbf H-(Az(x),z(x))_\mathbf H]-[F(z(x+h))-F(z(x))]\\
             =&(Az(x),\eta(x,h))_\mathbf H+\frac{1}{2}(A\eta(x,h),\eta(x,h))_\mathbf H\\
              &-(F'(z(x)),\eta(x,h))_\mathbf H+o(\|\eta(x,h)\|_{\mathbf H}).
  \end{align*}
  From \eqref{eq-uniform astimate of eta} we have
  \[
   a(x+h)-a(x)=(Az(x)-F'(z(x)),\eta(x,h))_\mathbf H+o(\|h\|_\mathbf H),\;\;\forall x\in \mathbf H^0,\; {\rm and}\;\|h\|_\mathbf H\to 0.
  \]
 Since $z^\pm(x)$ is the solution of \eqref{eq-zpm(x)} and from the definition of $\eta(x,h)$, we have
 \[
 (Az(x)-F'(z(x)),\eta(x,h))_\mathbf H=(Az(x)-F'(z(x)),h)_\mathbf H,\;\;\forall x, h\in \mathbf H^0,
 \]
  so we have
  \[
  a(x+h)-a(x)=(Az(x)-F'(z(x)),h)_\mathbf H+o(\|h\|_\mathbf H),\;\;\forall x\in \mathbf H^0,\; {\rm and}\;\|h\|_\mathbf H\to 0,
  \]
  and we have proved \eqref{eq-saddle point reduction-the derivative of a}.  If $F\in C^2(\mathbf H,\mathbb R)$, from \eqref{eq-zpm(x)} and by Implicit function theorem, we have $z^\pm \in C^1(\mathbf H^0,\mathbf H^\pm)$.
  From \eqref{eq-zpm(x)} and\eqref{eq-saddle point reduction-the derivative of a}, we have
  \[
  a'(x)=Ax-P^0F(z(x))
  \]
   and
  \[
  a''(x)=A|_{\mathbf H_0}-P^0F''(z(x))z'(x),
  \]
 that is to say $a\in C^2(\mathbf H^0,\mathbb R)$. Finally, from Theorem \ref{thm-relative morse index and spectral flow} received above, Definition 2.8 and Lemma 2.9 in \cite{Wang-Liu-2016}, we have \eqref{eq-relation between morse index and our index}. \endproof
\subsection{Some abstract critical points Theorems}
In this part, we will give some  abstract critical points Theorems for (O.E.) by the method of saddle point reduction introduced above. Since we have Proposition \ref{prop-relations between indexes}, we will not distinguish $i^*_A(B)$ from $i_A(B)$.  Beside condition ($F_1$), assume $F$ satisfying the following condition.\\
 \noindent ($F_2$) There exist $B_1,B_2\in \mathcal{L}_s(\mathbf H,-b,b)$ and $B:\mathbf H\to \mathcal{L}_s(\mathbf H,-b,b)$ satisfying
\[
B_1\leq B_2,\;i_A(B_1)=i_A(B_2),\; {\rm and}\; \nu_A(B_2)=0,
\]
 \[
 B_1\leq B(z) \leq B_2,\forall z\in\mathbf H,
 \]
 such that
\[
F'(z)-B(z)z=o(\|z\|_\mathbf H),\|z\|_\mathbf H\to\infty.
\]
Before the following Theorem, we need a Lemma.
\begin{lem}\label{lem-0 has a positive distance from sigma(A-B)}
Let $B_1,B_2\in \mathcal{L}_s(\mathbf H,-b,b)$ with $B_1\leq B_2,\;i_A(B_1)=i_A(B_2),\; {\rm and}\; \nu_A(B_2)=0$, then there exists $\varepsilon>0$, such that for all $B\in\mathcal{L}_s(\mathbf H)$ with
\[
 B_1\leq B \leq B_2,
\]
we have
\[
\sigma(A-B)\cap (-\varepsilon,\varepsilon)=\emptyset.
\]
\end{lem}
{\noindent}{\bf Proof.} For the property of $i_A(B)$, we have $\nu_A(B_1)=0$. So there is $\varepsilon>0$, such that
\[
i_A(B_{1,\varepsilon})=i_A(B_1)=i_A(B_2)=i_A(B_{2,\varepsilon}),
\]
with $B_{*,\varepsilon}=B_*+\varepsilon\cdot I,(*=1,2)$. Since  $B_{1,\varepsilon}\leq B-\varepsilon I<B+\varepsilon I\leq B_2'$. It follows that $i_A(B-\varepsilon I)=i_A(B+\varepsilon I)$. Note that
\[
i_A(B+\varepsilon)-i_A(B-\varepsilon)=\sum_{-\varepsilon < t \le  \varepsilon } \nu_A(B-t \cdot I).
\]
We have $0\notin \sigma(A-B-\eta),\;\forall \eta\in(-\varepsilon,\varepsilon)$, thus the proof is complete.\endproof

 \begin{thm}\label{thm-abstract thm 1 for the existence of solution}
Assume  $A\in\mathcal O^0_e(-b,b)$. If $F$ satisfies conditions ($F_1$) and ($F_2$), then (O.E.) has at least one solution.
 \end{thm}
 \noindent{\bf Proof.} Firstly, for $\lambda\in[0,1]$, consider the following equation
 \[
 Az=(1-\lambda) B_1z+\lambda F'(z).\eqno(O.E.)_\lambda
 \]
  We claim that the set of all the solutions ($z,\lambda$) of (O.E.)$_\lambda$ are a priori bounded. If not, assume there exist $\{(z_n,\lambda_n)\}$ satisfying (O.E.)$_\lambda$ with $\|z_n\|_{\mathbf H}\to\infty$. Without lose of generality, assume $\lambda_n\to\lambda_0\in[0,1]$. Denote by
\[
F_\lambda(z)=\frac{1-\lambda}{2}(B_1z,z)_{\mathbf H}+\lambda F(z),\;\forall z\in \mathbf H.
 \]
 Since $F$ satisfies condition ($F_1$) and $B_1\in \mathcal{L}_s(\mathbf H,-b,b)$, we have $F'_\lambda:\mathbf H\to\mathbf H$ is Lipschitz continuous with its Lipschitz constant less than $b$, that is to say there exists  $\hat{l}\in [\l_F,b)$ such that
 \[
 \|F'_\lambda(z+h)-F'_\lambda(z)\|_{\mathbf H}\leq \hat{l}\|h\|_{\mathbf H},\;\forall z,h\in\mathbf H,\lambda\in[0,1].
 \]
 Now, consider the projections defined in \eqref{eq-projections}, choose  $l\in (\hat{l},b)$ satisfying $-l,l\notin \sigma(A)$, from \eqref{eq-decomposition 1 of OE} and \eqref{eq-decomposition 2 of OE}, we decompose $z_n$ by
 \[
 z_n=z^+_n+z^-_n+x_n,
 \]
 with $z^*_n\in\mathbf H^*$($*=\pm,0$) and $z^{\pm}_n$ satisfies Proposition \ref{prop-continuous and property of saddle point reduction} with $l_F$ replaced by $\hat{l}$. So we have $\|x_n\|_\mathbf H\to\infty$. Denote by
 \[
 y_n=\frac{z_n}{\|z_n\|_\mathbf H},
 \]
 and $\bar{B}_n:=(1-\lambda_n)B_1+\lambda_nB(z_n)$, we have
\begin{equation}\label{eq-the equation of yn}
 Ay_n=\bar{B}_ny_n+\frac{o(\|z_n\|_\mathbf H)}{\|z_n\|_\mathbf H}.
 \end{equation}
Decompose $y_n=y^{\pm}_n+y^0_n$ with $y^*_n=z^*_n/\|z_n\|_\mathbf H$, we have
\begin{align*}
\|y^0_n\|_\mathbf H&=\frac{\|x_n\|_\mathbf H}{\|z_n\|_\mathbf H}\\
                   &\geq\frac{\|x_n\|_\mathbf H}{\|x_n\|_\mathbf H+\|z^+_n+z^-\|_\mathbf H}\\
                   &\geq\frac{(l-\hat{l})\|x_n\|_\mathbf H}{l\|x_n\|_\mathbf H+\|F'_\lambda(0)\|_\mathbf H}.
\end{align*}
That is to say
 \begin{equation}\label{eq-y0n not to 0}
\|y^0_n\|_\mathbf H\geq c>0
 \end{equation}
 for some constant $c>0$ and $n$ large enough.
 Since $B_1\leq B(z)\leq B_2$, we have $B_1\leq \bar{B}_n\leq B_2$.
 Let $\mathbf H=\mathbf H^+_{A-\bar B_n}\bigoplus \mathbf H^-_{A-\bar B_n}$ with $A-\bar B_n$ is positive and negative define on $\mathbf H^+_{A-\bar B_n}$ and $\mathbf H^-_{A-\bar B_n}$ respectively. Re-decompose $y_n=\bar{y}^+_n+\bar{y}^-_n$ respect to $\mathbf H^+_{A-\bar B_n}$ and $\mathbf H^-_{A-\bar B_n}$. From Lemma \ref{lem-0 has a positive distance from sigma(A-B)} and \eqref{eq-the equation of yn}, we have
 \begin{align}\label{eq-y0n to 0}
 \| y^0_n\|^2_\mathbf H&\leq  \| y_n\|^2_\mathbf H\nonumber\\
                     &\leq \frac{1}{\varepsilon} ((A-\bar{B}_n)y_n,\bar{y}^+_n+\bar{y}^-_n)_\mathbf H\nonumber\\
                     &\leq \frac{1}{\varepsilon} \frac{o(\|z_n\|_\mathbf H)}{\|z_n\|_\mathbf H}\|y_n\|_\mathbf H.
 \end{align}
 Since $\|z_n\|_\mathbf H\to \infty$ and $\|y_n\|=1$, we have $\|y^0_n\|_\mathbf H\to 0$ which  contradicts to \eqref{eq-y0n not to 0}, so we have $\{z_n\}$ is bounded.

 Secondly, we apply the topological degree theory to complete the proof. Since the solutions of (O.E.)$_\lambda$ are bounded, there is a number $R>0$ large eoungh,
 such that all of the solutions $z_\lambda$ of (O.E.)$_\lambda$ are in the ball $B(0,R):=\{z\in \mathbf H| \|z\|_\mathbf H<R\}$.  So we have the Brouwer degree
 \[
deg (a'_1,B(0,R)\cap \mathbf H^0,0)= deg (a'_0,B(0,R)\cap \mathbf H^0,0)\neq 0,
 \]
 where $a_\lambda(x)=\frac{1}{2}(Az_\lambda(x),z_\lambda(x))_\mathbf H-F_\lambda(z_\lambda(x))$, $\lambda\in[0,1]$. That is to say (O.E.) has at least one solution.
 \endproof

 In Theorem \ref{thm-abstract thm 1 for the existence of solution}, the non-degeneracy condition of $B(z)$ is important to keep the boundedness of the solutions. The following theorem will not need this non-degeneracy condition, the idea is from \cite{Ji-2018}.

 \begin{thm}\label{thm-abstract thm 3}
Assume  $A\in\mathcal O^0_e(-b,b)$. If $F$ satisfies conditions ($F_1$) and the following condition.\\
($F^\pm_2$) There exists $M>0$, $B_\infty\in\mathcal{L}_s(\mathbf H, -b,b)$, such that
\[
F'(z)=B_\infty z+r(z),
\]
with
\[
\|r(z)\|_\mathbf H\leq M,\;\;\forall z\in\mathbf H,
\]
and
\begin{equation}\label{eq-condition of r in ab-thm 3}
(r(z),z)_\mathbf H\to\pm\infty,\;\;\|z\|_\mathbf H\to\infty.
\end{equation}
 Then (O.E.) has at least one solution.
\end{thm}
\pf If $0\not\in \sigma(A-B_\infty)$, then with the similar method in Theorem \ref{thm-abstract thm 1 for the existence of solution}, we can prove the result.
So we assume $0\in \sigma(A-B_\infty)$ and  we only consider the case of ($F^-_2$).   Since $0$ is an isolate eigenvalue of $A-B_\infty$ with finite dimensional eigenspace (see \cite{Wang-Liu-2016} for details), there exists $\eta>0$ such that
\[
(-\eta,0)\cap\sigma(A-B_\infty)=\emptyset.
\]
 For any $\varepsilon\in (0,\eta)$, we have
  $0\not\in \sigma(\varepsilon+A-B_\infty)$. Thus, with the similar method in Theorem \ref{thm-abstract thm 1 for the existence of solution},
we can prove that there exists  $z_\varepsilon\in \mathbf H$ satisfying the following equation
\begin{equation}\label{eq-equation of z-varepsilon}
\varepsilon z_\varepsilon+(A-B_\infty)z_\varepsilon=r(z_\varepsilon).
\end{equation}
In what follows, We divide the following proof into two steps and  $C$  denotes various constants independent of $\varepsilon$.

{\bf Step 1. We claim that  $\|z_\varepsilon\|_\mathbf H\leq C$. } Since $z_\varepsilon$ satisfies the above equation,  we have
\begin{align*}
\varepsilon (z_\varepsilon,z_\varepsilon)_\mathbf H&=-((A-B_\infty)z_\varepsilon,z_\varepsilon)_\mathbf H+(r(z_\varepsilon),z_\varepsilon)_\mathbf H\\
                                                   &\leq \frac{1}{\eta}\|(A-B_\infty)z_\varepsilon\|^2_\mathbf H+M\|z_\varepsilon\|_\mathbf H\\
                                                   &=\frac{1}{\eta}\|\varepsilon z_\varepsilon-r(z_\varepsilon)\|^2_\mathbf H+M\|z_\varepsilon\|_\mathbf H\\
                                                   &\leq \frac{\varepsilon^2}{\eta}\|z_\varepsilon\|^2_\mathbf H+C\|z_\varepsilon\|_\mathbf H +C.
\end{align*}
  So we have
\[
\varepsilon\|z_\varepsilon\|_\mathbf H\leq C.
\]
 Therefore
\begin{equation}\label{eq-boundedness of z-1}
\|(A-B_\infty)z_\varepsilon\|_\mathbf H=\|\varepsilon z_\varepsilon-r(z_\varepsilon)\|_\mathbf H\leq C.
\end{equation}
Now, consider the orthogonal splitting as defined in \eqref{eq-decomposition of space H},
\[
\mathbf H=\mathbf H^0_{A-B_\infty}\oplus\mathbf H^*_{A-B_\infty},
\]
where $A-B_\infty$ is zero definite on $\mathbf H^0_{A-B_\infty}$,  $\mathbf H^*_{A-B_\infty}$ is the orthonormal complement space of $\mathbf H^0_{A-B_\infty}$. Let $z_\varepsilon=u_\varepsilon+v_\varepsilon$ with $u_\varepsilon\in \mathbf H^0_{A-B_\infty}$ and $v_\varepsilon\in \mathbf H^*_{A-B_\infty}$. Since $0$ is an isolated point in $\sigma(A-B_\infty)$, from \eqref{eq-boundedness of z-1}, we have
\begin{equation}\label{eq-boundedness of z-2}
\|v_\varepsilon\|_\mathbf H\leq C
\end{equation}
Additionally, since $r(z)$ and $v_\varepsilon$ are bounded, we have
\begin{align}\label{eq-boundedness of z-3}
(r(z_\varepsilon),z_\varepsilon)_\mathbf H&=(r(z_\varepsilon),v_\varepsilon)_\mathbf H+(r(z_\varepsilon),u_\varepsilon)_\mathbf H\nonumber\\
                                           &=(r(z_\varepsilon),v_\varepsilon)_\mathbf H+(\varepsilon z_\varepsilon+(A-B_\infty)z_\varepsilon,u_\varepsilon)_\mathbf H\nonumber\\
                                                                                      &=(r(z_\varepsilon),v_\varepsilon)_\mathbf H+\varepsilon(u_\varepsilon,u_\varepsilon)_\mathbf H\nonumber\\
                                           &\geq C.
\end{align}
Therefor, from \eqref{eq-condition of r in ab-thm 3},  $\|u_\varepsilon\|_\mathbf H$ are bounded in $\mathbf H$ and we have proved the boundedness of $\|z_\varepsilon\|_\mathbf H$.

{\bf Step 2. Passing to a sequence of $\varepsilon_n\to 0$,  there exists $z\in\mathbf H$
such that
\[
\displaystyle\lim_{\varepsilon_n\to 0}\|z_{\varepsilon_n}-z\|_\mathbf H=0.
\]
}
Different from the above splitting, now, we recall the projections $P^-_A,\;P^0_A$ and $P^+_A$ defined in \eqref{eq-projections} and the splitting $\mathbf H=\mathbf H^-\oplus\mathbf H^0\oplus\mathbf H^+$ with $\mathbf H^*=P^*_A$($*=\pm,0$).
So $z_\varepsilon$ has the corresponding splitting
\[
z_\varepsilon=z_\varepsilon^++z_\varepsilon^-+z_\varepsilon^0,
\]
 with $z_\varepsilon^*\in \mathbf H^*$ respectively.
Since $\mathbf H^0$  is a finite dimensional space and $\|z_\varepsilon\|_\mathbf H\leq C$,  there exists a sequence $\varepsilon_n\to 0$ and $z^0\in \mathbf H^0$, such that
\[
\displaystyle\lim_{n\to\infty}z^0_{\varepsilon_n}=z^0.
\]
 For simplicity, we rewrite $z^*_n:=z^*_{\varepsilon_n}$, $A_n:=\varepsilon_n+A$ and $A^\pm_n:=A_n|_{\mathbf H^\pm} $.
Since  $z_\varepsilon$ satisfies \eqref{eq-equation of z-varepsilon}, we have
\[
z^\pm_n=(A_n^\pm)^{-1}P^\pm_A F'(z^+_n+z^-_n +z^0_n).
\]
 Since $F$ satisfies ($F_1$), with the similar method used in Proposition \ref{prop-continuous and property of saddle point reduction}, for $n$ and $m$ large enough, we have
 \begin{align*}
\|z^\pm_n-z^\pm_m\|_\mathbf H=&\|(A_n^\pm)^{-1}P^\pm_A F'(z_n)-(A_m^\pm)^{-1}P^\pm_A F'(z_m)\|_\mathbf H \\
\leq&\|(A_n^\pm)^{-1}P^\pm_A (F'(z_n)-F'(z_m))\|_\mathbf H+\|((A_n^\pm)^{-1}-(A_m^\pm)^{-1})P^\pm_A F'(z_m)\|_\mathbf H \\
\leq&\frac{l_F}{l}\|z_n-z_m\|_\mathbf H+\|((A_n^\pm)^{-1}-(A_m^\pm)^{-1})P^\pm_A F'(z_m)\|_\mathbf H.
\end{align*}
Since $(A_n^\pm)^{-1}-(A_m^\pm)^{-1}=(\varepsilon_m-\varepsilon_n)(A_n^\pm)^{-1}(A_m^\pm)^{-1}$ and $z_n$ are bounded in $\mathbf H$, we have
\[
\|((A_n^\pm)^{-1}-(A_m^\pm)^{-1})P^\pm_A F'(z_m)\|_\mathbf H=o(1),\;\;n,m\to\infty.
\]
So we have
\[
\|z^\pm_n-z^\pm_m\|_\mathbf H\leq \frac{l_F}{l-l_F}\|z^0_n-z^0_m\|_\mathbf H+o(1),\;\;n,m\to\infty,
\]
therefor, there exists $z^\pm\in\mathbf H^\pm$, such that $\displaystyle\lim_{n\to\infty}\|z^\pm_n- z^\pm\|_\mathbf H=0$. Thus, we have
\[
\displaystyle\lim_{n\to\infty}\|z_{\varepsilon_n}-z\|_\mathbf H=0,
\]
with $z=z^-+z^++z^0$.
Last, let $n\to\infty$ in \eqref{eq-equation of z-varepsilon}, we have $z$ is a solution of (O.E.).\endproof
\begin{thm}\label{thm-abstract thm 2 for the multiplicity of solutions}
 Assume  $A\in\mathcal O^0_e(-b,b)$, $F$  satisfies ($F_1$) with $\pm l_F\not\in\sigma(A)$ and   the following condition:\\
($F^+_3$) There exist $B_3\in\mathcal{L}_s(\mathbf{H},-b,b)$ and $C\in\mathbb{R}$, such that
\[
B_3>\beta:=\max\{\lambda|\lambda\in \sigma_A\cap(-\infty,l_F)\},
\]
with
\[
F(z)\geq\frac{1}{2}(B_3z,z)_\mathbf{H}-C,\;\;\forall z\in \mathbf H.
\]
Or ($F^-_3$) There exist $B_3\in\mathcal{L}_s(\mathbf{H},-b,b)$ and $C\in\mathbb{R}$, such that
\[
B_3<\alpha:=\min\{\lambda|\lambda\in \sigma_A\cap(-l_F,\infty)\},
\]
with
\[
F(z)\leq\frac{1}{2}(B_3z,z)_\mathbf{H}+C,\;\;\forall z\in \mathbf H.
\]
Then (O.E.) has  at least one solution. Further more, assume $F$ satisfies \\
($F^\pm_4$)  $F\in C^2(\mathbf H,\mathbb{R})$, $F'(0)=0$ and there exists $B_0\in\mathcal{L}_s(\mathbf{H},-b,b)$ with
\begin{equation}\label{eq-twisted condition 1 in abstract thm 2}
\pm(i_A(B_0)+\nu_A(B_0))<\pm i_A(B_3),
\end{equation}
such that
\[
F'(z)=B_0z+o(\|z\|_\mathbf H),\;\;\|z\|_\mathbf H\to 0.
\]
Then (O.E.) has at least one nontrivial solution. Additionally, if
\begin{equation}\label{eq-twisted condition 2 in abstract thm 2}
\nu_A(B_0)=0
\end{equation}
then (O.E.) has at least two nontrivial solutions.\end{thm}
\pf  We only consider the case of ($F^+_3$). According to the saddle point reduction, since $\pm l_F\not\in\sigma(A)$, we can choose $l\in(l_F,b)$ in \eqref{eq-projections} satisfying
\[
[-l,-l_F]\cap\sigma(A)=\emptyset=[l_F,l]\cap\sigma(A).
\]
We turn to the function
\[
a(x)=\frac{1}{2}(Az(x),z(x))-F(z(x)),
\]
where $z(x)=x+z^+(x)+z^-(x)$, $x\in \mathbf H^0$ and $z^\pm\in \mathbf H^\pm$. Denote by $w(x)=x+z^-(x)$ and write $z=z(x)$, $w=w(x)$ for simplicity. Since
\begin{equation}\label{eq-eq 1 in abstract thm 2}
a(x)=\left\{\frac{1}{2}(Aw,w)-F(w)\right\}+\left\{\frac{1}{2}[(Az,z)-(Aw,w)]-[F(z)-F(w)]\right\}.
\end{equation}
 By condition ($F^+_3$), we obtain
\begin{equation}\label{eq-eq 2 in abstract thm 2}
\frac{1}{2}(Aw,w)-F(w)\leq \frac{1}{2}((\beta-B_3)w,w)_\mathbf{H}+C,
\end{equation}
and  the terms in the second bracket are equal to
\begin{align}\label{eq-eq 3 in abstract thm 2}
&\frac{1}{2}(Az^+,z^+)-\int^{1}_{0}(F'(sz^++w),z^+)ds\nonumber\\
=&\frac{1}{2}(Az^+,z^+)-(F'(z^++w),z^+)+\int^{1}_{0}(F'(z^++w)-F'(sz^++w),z^+)ds\nonumber\\
=&-\frac{1}{2}(Az^+,z^+)+\int^{1}_{0}(F'(z^++w)-F'(sz^++w),z^+)ds\nonumber\\
\leq&-\frac{1}{2}(Az^+,z^+)+\int^{1}_{0}(1-s)ds\cdot l_F\cdot\|z^+\|^2_\mathbf H\nonumber\\
\leq& -\frac{l-l_F}{2}\|z^+\|^2_\mathbf H,
\end{align}
where the last equality is from the fact that $Az^+=P^+F'(z^++w)$. From \eqref{eq-eq 1 in abstract thm 2},\eqref{eq-eq 2 in abstract thm 2} and \eqref{eq-eq 3 in abstract thm 2} we have

\begin{align*}
a(x)&\leq \frac{1}{2}((\beta-B_3)w,w)_\mathbf{H} -\frac{l-l_F}{2}\|z^+\|^2_\mathbf H+C\\
    &\to-\infty,\; as \|x\|\to\infty.
\end{align*}
 Thus the function $-a(x)$ is bounded from below and satisfies the (PS) condition. So the maximum of $a$ exists and the maximum points are critical points of $a$.

 In order to prove the second part, similarly, we only consider the case of ($F^+_3$) and ($F^+_4$). We only need to realize that $0$ is not a maximum point from \eqref{eq-twisted condition 1 in abstract thm 2}, so the maximum points discovered above are not $0$. In the last, if \eqref{eq-twisted condition 2 in abstract thm 2} is satisfied, we can use the classical three critical points theorem, since $0$ is neither a maximum nor degenerate and the proof is complete. \endproof

\begin{rem}
(A). Theorem \ref{thm-abstract thm 2 for the multiplicity of solutions} is generalized from \cite[IV,Theorem 2.3]{Chang-1993}. In the first part of our Theorem, we do not need $F$ to be $C^2$ continuous.\\
(B). Theorem \ref{thm-abstract thm 2 for the multiplicity of solutions}  is different from our former result in \cite[Theorem 3.6]{Wang-Liu-2016}.
Here, we need the Lipschitz condition to keep the method of saddle point reduction valid, where, in \cite[Theorem 3.6]{Wang-Liu-2016}, in order to use the method of dual variation, we need the convex property.
\end{rem}
 \section{Applications in one dimensional wave equation}\label{section-applications}
In this section,  we will consider the following one dimensional wave equation
\[
\left\{\begin{array}{ll}
       \Box u\equiv u_{tt}-u_{xx}=f(x, t, u),\\
                  u(0,t)=u(\pi, t)= 0, \\
                         u(x, t+T)=u(x, t),\\
       \end{array}
\right.\forall (x, t)\in[0,\pi]\times S^1, \eqno(W.E.)
\]
where $T>0$, $S^1:=\mathbb{R}/T\mathbb{Z}$ and $f:[0,\pi]\times S^1\times\mathbb{R}\to\mathbb{R}$.
In what follows we assume systematically that $T$ is a rational multiple of $\pi$.
So, there exist coprime integers $(p,q)$, such that
\[
 T=\frac{2\pi q}{p}.
\]
Let
\[
 L^2:=\left\{u,u=\sum_{j\in\mathbb{N}^+,k\in\mathbb{Z}} u_{j,k}\sin jx\exp ik\frac{p}{q}t\right\},
\]
where $i=\sqrt{-1}$ and $u_{j,k}\in \mathbb{C}$ with $u_{j,k}=\bar{u}_{j,-k}$, its inner product is
\[
(u,v)_2=\sum_{j\in\mathbb{N}^+,k\in\mathbb{Z}}(u_{j,k},\bar{v}_{j,k}),\;u,v\in L^2,
\]
the corresponding norm is
\[
 \|u\|^2_2=\sum_{j\in\mathbb{N}^+,k\in\mathbb{Z}}|u_{j,k}|^2\;u,v\in L^2.
\]

Consider $\Box$ as an unbounded self-adjoint operator on $L^2$. Its' spectrum set is
\[
\sigma(\Box)=\{(p^2k^2-q^2j^2)/q^2|j\in\mathbb{N}^+,k\in\mathbb{Z}\}.
\]
It is easy to see $\Box$ has only one essential spectrum $\lambda_0=0$. Let $\Omega:=[0,\pi]\times S^1$, assume $f$ satisfying the following conditions.

\noindent($f_1$) $f\in C(\Omega\times\mathbb{R},\mathbb{R})$, there exist $ b\neq 0$ and $l_F\in(0,|b|)$, such that
\[
|f_{ b}(x,t,u+v)-f_{ b}(x,t,u)|\leq l_F|v|,\;\;\forall (x,t)\in \Omega,\;u,v\in\mathbb{R},
\]
where
\[
f_{ b}(x,t,u):=f(x,t,u)-bu,\;\;\forall (x,t,u)\in\Omega\times \mathbb{R}.
\]
Let the working space $\mathbf H:=L^2$ and the operator $A:=\Box-b\cdot I$, with $I$ the identity map on $\mathbf H$. Thus we have $A\in\mathcal O^0_e(-|b|,|b|)$. Denote $L^\infty:=L^\infty(\Omega, \mathbb{R})$  the set of all essentially bounded functions. For any $g\in L^\infty$, it is easy to see $g$ determines a bounded self-adjoint
operator on $L^2$, by
\[
u(x,t)\mapsto g(x,t)u(x,t),\;\;\forall u\in L^2,
\]
without confusion, we still denote this operator by $g$, that is to say we have the continuous embedding $L^\infty\hookrightarrow \mathcal L_s(\mathbf H)$. Thus for any $g\in L^\infty\cap \mathcal L_s(\mathbf H,-|b|,|b|)$,  we have the index pair ($i_A(g),\nu_A(g)$). Besides, for any $g_1,g_2\in L^\infty$, $g_1\leq g_2$ means that
\[
g_1(x,t)\leq g_2(x,t),\;{\rm a.e.} (x,t)\in\Omega.
\]
\noindent ($f_2$) There exist $g_1,g_2\in L^\infty\cap \mathcal L_s(\mathbf H,-|b|,|b|)$ and $g\in L^\infty(\Omega\times \mathbb{R},\mathbb{R})$, with
\[
g_1\leq g_2,\;i_A(g_1)=i_A(g_2),\;\nu_A(g_2)=0,
\]
\[
g_1(x,t)\leq g(x,t,u)\leq g_2(x,t),\;\;\;{\rm a.e.} (x,t,z)\in\Omega\times\mathbb{R},
\]
such that
\[
f_{ b}(x,t,u)-g(x,t,u)u=o(|u|),\;|u|\to\infty,\;{\rm uniformly for }(x,t)\in \Omega.
\]

We have the following results.
\begin{thm}\label{thm-application-1}
Assume  $T$ is a rational multiple of $\pi$, $f$ satisfying  ($f_1$) and ($f_2$), then (W.E.) has a weak solution.
\end{thm}
\noindent{\bf Proof of Theorem \ref{thm-application-1}.} Let
\[
\mathcal{F}_b(x,t,u):=\int^u_0f_b(x,t,s)ds,\;\;\forall(x,t,u)\in\Omega\times\mathbb{R},
\]
and
\begin{equation}\label{eq-definition of F}
F(u):=\int_\Omega \mathcal F_b(x,t,u(x,t))dxdt,\;\;\forall u\in \mathbf H.
\end{equation}
 It is easy to verify that $F$ will satisfies condition ($F_1$) and ($F_2$) if $f$ satisfies condition ($f_1$) and ($f_2$).
 Thus, by Theorem \ref{thm-abstract thm 1 for the existence of solution}, the proof is complete.\endproof

 Here, we give an example of Theorem \ref{thm-application-1}.
\begin{eg}
For any $b\neq 0$, assume $\alpha,\beta\in (-|b|,|b|)$ and
$
[\alpha,\beta]\cap \sigma(\Box-b)=\emptyset.
$
Let
\[
g(x,t,u):=\displaystyle\frac{\beta-\alpha}{2}\sin[\varepsilon_1\ln (|x|+|t|+|u|+1)]+\frac{\alpha+\beta}{2},
\]
and $h\in C(\mathbb{R},\mathbb{R})$ is Lipschitz continuous with
\[
h(u)=o(|u|),\;\;|u|\to\infty.
\]
 then
\[
f(x,t,u):=bu+g(x,t,u)u+\varepsilon_2h(u)
\]
will satisfies condition ($f_1$) and ($f_2$) for $\varepsilon_1$ and $\varepsilon_2>0$ small enough.
\end{eg}

\begin{thm}\label{thm-application-3}
Assume  $T$ is a rational multiple of $\pi$, $f$ satisfies  ($f_1$) and the following condition,\\
($f^\pm_2$) There exists $g_\infty(x,t)\in L^\infty\cap\mathcal{L}_s(\mathbf H, -|b|,|b|)$ with
\[
|f_b(x,t,u)-g_{\infty}(x,t)u|\leq M_{1}\;\;\forall (x,t,u)\in\Omega\times\mathbf R,
\]
and
\begin{equation}\label{eq-condition of r}
\pm(f_b(x,t,u)-g_{\infty}(x,t)u,u)\geq c|u|,\;\;\forall (x,t,u)\in\Omega\times\mathbb R/[-M_{2},M_{2}],
\end{equation}
where $M_1,\;M_2,\;c>0$ are constants.  Then (W.E.) has a weak solution.
\end{thm}
\pf We only consider the case of $f^{-}_{2}$. Let $r(x,t,u):=f_b(x,t,u)-g_{\infty}(x,t)u$, then $r$ is bounded in $\mathbf H$. Generally speaking, from \eqref{eq-condition of r}, we cannot prove \eqref{eq-condition of r in ab-thm 3}, so we cannot use Theorem \ref{thm-abstract thm 3} directly. By checking the proof of Theorem \ref{thm-abstract thm 3}, in step 1, when we got \eqref{eq-boundedness of z-2}, \eqref{eq-condition of r in ab-thm 3} was only used to get the boundedness of  $z^0_{\varepsilon}$. Now, with \eqref{eq-condition of r}, we can also get the boundedness of $z^0_{\varepsilon}$ from \eqref{eq-boundedness of z-2}. Recall that $\mathbf H=L^{2}(\Omega)$ in this section, from the boundedness of $z^{\pm}_{\varepsilon}$ in $\mathbf H$, we have the boundedness of $z^{\pm}_{\varepsilon}$ in $L^{1}(\Omega)$. On the other hand, since $\ker (A-g_{\infty})$ is a finite dimensional space, if $\|z^{0}_{\varepsilon}\|_{\mathbf H}\to\infty$, we have $\|z^{0}_{\varepsilon}\|_{L^{1}}\to\infty$, thus $\|z_{\varepsilon}\|_{L^{1}}\to\infty$. Therefor, we have the contradiction from \eqref{eq-boundedness of z-2} and \eqref{eq-condition of r}. So we have gotten the boundedness of $z^0_{\varepsilon}$. The rest part of the proof  is similar to Theorem \ref{thm-abstract thm 3}, we omit it here.

\begin{eg}
Here we give an example of Theorem \ref{thm-abstract thm 3}. For any $b\neq 0$, and $g_\infty\in C(\Omega)$ with
\[
\|g_\infty\|_{C(\Omega)}<|b|.
\]
Let $r(u)=\varepsilon\arctan u$, then
\[
f(x,t,u):=bu+g_\infty(x,t)u\pm r(u)
\]
will satisfies the conditions in Theorem \ref{thm-abstract thm 3} for $\varepsilon>0$ small enough.
\end{eg}

Now, in order to use Theorem \ref{thm-abstract thm 2 for the multiplicity of solutions}, we assume $f$ satisfies the following conditions.

\noindent ($f^{\pm}_3$) There exists $g_3(x,t)\in L^\infty\cap\mathcal{L}_s(\mathbf H, -|b|,|b|)$,  with
\[
\pm g_3(x,t)>\max\{\lambda|\lambda\in \sigma_{(\pm A)}\cap (-\infty,l_F)\},
\]
such that
\[
\pm\mathcal{F}_b(x,t,u)\geq \frac{1}{2}(g_3(x,t)u,u)+c,\;\;\forall (x,t,u)\in\Omega\times\mathbb{R},
\]
for some $c\in\mathbb{R}$.

\noindent ($f^\pm_4$) $f\in C^1(\Omega\times \mathbb{R},\mathbb{R})$, $f(x,t,0)\equiv 0,\;\forall(x,t)\in\Omega$ and
\[
g_0(x,t):=f'_b(x,t,u),\;\;\forall (x,u)\in\Omega,
\]
with
\[
\pm (i_A(g_0)+\nu_A(g_0))<\pm i_A(g_3).
\]
We have the following result.
\begin{thm}\label{thm-application-2}
Assume  $T$ is a rational multiple of $\pi$. \\
(A.)If $f$ satisfies condition ($f_1$), ($f^+_3$)( or ($f^-_3$)), then  (W.E.) has at least one solution.\\
(B.) Further more, if $f$ satisfies condition ($f^+_4$)( or ($f^-_4$)), then (W.E.) has at least one nontrivial solution. Additionally, if $\nu_A(g_0)=0$, then (W.E.) has at least two nontrivial solutions.
\end{thm}
The proof is to verify the conditions in Theorem \ref{thm-abstract thm 2 for the multiplicity of solutions}, we only verify the smoothness of $F(u)$ defined in \eqref{eq-definition of F}. From condition ($f_1$) and $f\in C^1(\Omega\times\mathbb{R})$, we have the derivative $f'_b(x,t,u)$ of $f_b$ with respect to $u$, satisfying
\begin{equation}\label{eq-the boundedness of f'_b}
|f'_b(x,t,u)|\leq l_F,\;\;\forall (x,t,u)\in\Omega\times\mathbb{R}.
\end{equation}
For any $u,v\in\mathbf{H}$,
\begin{align*}
F'(u+v)-F'(u)&=f_b(x,t,u+v)-f_b(x,t,u)\\
             &=f'_b(x,t,u)v+(f'_b(u+\xi v)-f'_b(u))v.
\end{align*}
From \eqref{eq-the boundedness of f'_b}, we have $f'_b(u+\xi v)-f'_b(u)\in \mathbf{H}$ and
\[
\displaystyle\lim_{\|v\|_\mathbf{H}\to 0}\|f'_b(u+\xi v)-f'_b(u)\|_\mathbf{H}=0,\;\;\forall u\in\mathbf{H}.
\]
That is to say $F''(u)=f'_b(x,t,u)$ and $F\in C^2(\mathbf{H},\mathbb{R})$.

\begin{eg}
In order to give an example for Theorem \ref{thm-application-2}, assume
\begin{equation}\label{eq-spectrum of Box}
\sigma(\Box)=\mathop{\cup}\limits_{n\in\mathbb{Z}}\{\lambda_n\},
\end{equation}
with $\lambda_0=0$ and $\lambda_n<\lambda_{n+1}$ for all $n\in\mathbb{Z}$. Choose any $k\in\{2,3\cdots\}$. Let
\[
g_0(x,t)\in C(\Omega,[\alpha,\beta]),\;{\rm with}\;[\alpha,\beta]\in(0,\lambda_k),
\]
and $h\in C(\mathbb{R},\mathbb{R})$ defined above. Define
\[
g(x,t,u):=g_0(x,t)+\displaystyle(\lambda_k-g_0(x,t)-\varepsilon_1)\frac{2}{\pi}\arctan(\varepsilon_1u^2),
\]
then
\[
f(x,t,u):=g(x,t,u)u+\varepsilon_2 h(u)
\]
will satisfies condition ($f_1$) and ($f^+_3$) with $b=\frac{\lambda_{k}}{2}$ and $\varepsilon_1,\varepsilon_2>0$ small enough.
Further more, if $g_0,h$ are $ C^1$ continuous and $\beta<\lambda_{k-1}$, we have condition ($f^+_4$) is satisfied.
Additionally, if   $[\alpha,\beta]\cap\sigma(\Box)=\emptyset$, then $\nu_A(g_0)=0$.
\end{eg}

\begin{rem}
We can also use Theorem \ref{thm-abstract thm 1 for the existence of solution} , Theorem \ref{thm-abstract thm 3} and Theorem \ref{thm-abstract thm 2 for the multiplicity of solutions} to consider the  radially symmetric solutions for the $n$-dimensional wave equation:
\[
\left\{\begin{array}{ll}
       \Box u\equiv u_{tt}-\vartriangle_x u=h(x,t,u), &t\in\mathbb{R},\;x\in B_R,\\
                  u(x,t)= 0, t\in\mathbb{R},        &t\in\mathbb{R},\;x\in \partial B_R,\\
                         u(x,t+T)=u(x,t),           &t\in\mathbb{R},\;x\in B_R,\\
       \end{array}
\right. \eqno(n \textendash W.E.)
\]
where $B_R=\{x\in\mathbb{R}^n,|x|<R\}$, $\partial B_R=\{x\in\mathbb{R}^n,|x|=R\}$, $n>1$ { and the nonlinear term $h$ is $T$-periodic in variable $t$}.
Restriction of the radially symmetry allows us to know the nature of spectrum of the wave operator.
Let $r=|x|$ and {$S^1:=\mathbb{R}/T$, if $h(x,t,u)=h(r,t,u)$} then the $n$-dimensional wave equation ($n$\textendash W.E.) can be transformed into:
\[
 \left\{\begin{array}{ll}
        A_0u:=u_{tt}-u_{rr}-\frac{n-1}{r}u_r=h(r,t,u),\\
         u(R,t)=0,\;                                     \\
         u(r,0)=u(r,T),\;u_t(r,0)=u_t(r,T),
        \end{array}
\right. \;\;\;(r,t)\in{\Omega:=[0,R]\times S^1}.\eqno(RS\textendash W.E.)
\]
$A_0$ is symmetric on $L^2(\Omega,\rho)$, where $\rho=r^{n-1}$ and
\[
 L^2(\Omega,\rho):=\left\{u|\|u\|^2_{L^2(\Omega,\rho)}:=\int_\Omega|u(t,r)|^2r^{n-1}dtdr<\infty\right\}.
\]
 By the asymptotic properties
of the Bessel functions (see\cite{Watson-1952}), the spectrum of the wave operator can be characterized (see\cite[Theorem 2.1]{Schechter-1998}).
Under some more assumption, the self-adjoint extension of $A_0$ has no  essential spectrum, and we can get more solutions of (RS\textendash W.E.).
\end{rem}

 \end{document}